\documentclass[11pt,a4j]{article}

\usepackage{lineno,hyperref}
 \usepackage{amsmath,amssymb,amsfonts}
 \usepackage{amsthm}
 \usepackage{latexsym}
 \usepackage{array}

\theoremstyle{definition} 
\newtheorem{theorem}{Theorem}[section]
\newtheorem{corollary}[theorem]{Corollary}
\newtheorem{prop}[theorem]{Proposition}
\newtheorem{lemma}[theorem]{Lemma}
\newtheorem{rem}{Remark}
\newtheorem{example}{Example}
\newtheorem{defn}{Definition}

\newcolumntype{I}{!{\vrule width 1.5pt}}
\newlength\savedwidth
\newcommand{\wcline}[1]{\noalign{\global\savedwidth\arrayrulewidth\global\arrayrulewidth 2pt} \cline{#1}
\noalign{\global\arrayrulewidth\savedwidth}}

\allowdisplaybreaks[1] 

\title{Pseudo-Poisson Nijenhuis manifolds}
\author{Tomoya Nakamura  \thanks{Department of Mathematics, Waseda University, \textup{3-4-1}, Okubo, Shinjuku-ku, Tokyo, Japan}
 \\ email: \href{mailto:x-haze@ruri.waseda.jp}{x-haze@ruri.waseda.jp}}
\date{\today}

\begin{document}

  \maketitle
  
  \begin{abstract}
  We introduce the notion of pseudo-Poisson Nijenhuis manifolds. These manifolds are generalizations of Poisson Nijenhuis manifolds defined by Magri and Morosi \cite{MM}. We show that there is a one-to-one correspondence between the pseudo-Poisson Nijenhuis manifolds and some quasi-Lie bialgebroid structures on the tangent bundle as in the case of Poisson Nijenhuis manifolds by Kosmann-Schwarzbach \cite{K}. For that reason, we expand the general theory of the compatibility of a $2$-vector field and a $(1,1)$-tensor. In the case of pseudo-Poisson Nijenhuis structures having some ``nondegeneracy'', we call structures corresponding to such structures pseudo-symplectic Nijenhuis structures, and investigate properties of those. In particular, we show that those structures induce twisted Poisson structures \cite{SW}.
  \end{abstract}

 \section{Introduction}
Poisson Nijenhuis structures were defined by Magri and Morosi \cite{MM} to study bi-Hamiltonian systems. A pair of a Poisson structure $\pi $ and a Nijenhuis structure $N$ on a $C^\infty$-manifold $M$ is a Poisson Nijenhuis structure on $M$ if those have some compatibility condition. It is known that Poisson Nijenhuis manifolds (i.e., manifolds with Poisson Nijenhuis structures) are related with various mathematical objects \cite{MM}, \cite{KMa}, \cite{K}.

Kosmann-Schwarzbach \cite{K} showed that there is a one-to-one correspondence between the Poisson Nijenhuis manifolds $(M, \pi ,N)$ and the Lie bialgebroids $((TM)_N,(T^*M)_\pi )$, where $(TM)_N$ is a Lie algebroid deformed by the Nijenhuis structure $N$ and $(T^*M)_\pi $ is the cotangent bundle equipped with the standard Lie algebroid structure induced by the Poisson structure $\pi $. On the other hand, Sti$\acute{\mathrm{e}}$non and Xu \cite{SX} introduced the concept of a Poisson quasi-Nijenhuis manifold $(M, \pi ,N,\phi )$, and showed that a Poisson quasi-Nijenhuis manifold corresponded to a quasi-Lie bialgebroid $((T^*M)_\pi , d_N ,\phi )$. Here a Lie bialgebroid \cite{MX}, \cite{MX2} consists of a pair $(A,A^*)$, where $A$ is a Lie algebroid, and $A^*$ is the dual bundle equipped with a Lie algebroid structure, together with the following condition: for any $D_1$ and $D_2$ in $\Gamma (\Lambda ^*A)$,
\begin{eqnarray}\label{Gerstenhaber}
d_*[D_1,D_2]_A=[d_* D_1, D_2]_A+(-1)^{\mathrm{deg}D_1+1}[D_1, d_* D_2]_A,
\end{eqnarray}
where a bracket $[\cdot ,\cdot ]_A$ is the Schouten bracket of the Lie bracket of $A$, and $d_*$ is the Lie algebroid differential determined from the Lie algebroid structure of $A^*$ \cite{LWX}. Since the Lie algebroid structure on $A^*$ can be recovered from the derivation $d_*$, a Lie bialgebroid $(A,A^*)$ is also denoted by $(A,d_*)$. A quasi-Lie bialgebroid \cite{ILX} is a Lie algebroid $(A, [\cdot ,\cdot ]_A, a)$ equipped with a degree-one derivation $d_*$ of the Gerstenhaber algebra $(\Gamma (\Lambda ^*A), \wedge ,[\cdot ,\cdot]_A)$, i.e., $d_*$ satisfies (\ref{Gerstenhaber}), and a 3-section of $A$, $\phi _A$ in $\Gamma (\Lambda ^3A)$ such that $d_*^2=[\phi _A, \cdot ]_A$ and $d_*\phi _A=0$.

Our main purposes in this paper are to define a pseudo-Poisson Nijenhuis manifold $(M, \pi ,N,\Phi )$ and to show that there is a one-to-one correspondence between the pseudo-Poisson Nijenhuis manifolds $(M, \pi ,N,\Phi )$ and the quasi-Lie bialgebroids $((TM)_N, d_\pi ,\Phi )$. A quasi-Lie bialgebroid $((TM)_N, d_\pi ,\Phi )$ is, so to speak, ``the opposite side'' of a quasi-Lie bialgebroid $((T^*M)_\pi , d_N ,\phi )$.
Here $d_N$ and $d_\pi $ are operators of $\Omega ^*(M):=\Gamma (\Lambda ^*T^*M)$ and $\mathfrak{X}^*(M):=\Gamma (\Lambda ^*TM)$ determined from a $2$-vector field $\pi $ and a $(1,1)$-tensor $N$, respectively. A {\it pseudo-Poisson Nijenhuis structure} on $M$ is a triple consisting of a $2$-vector field $\pi $ which does not need to be a Poisson structure, a Nijenhuis structure $N$ ``compatible'' with $\pi $ and a $3$-vector field $\Phi $ with conditions
\begin{eqnarray*}
&\mathrm{(i)}& [\pi ,\Phi ]=0,\\
&\mathrm{(ii)}& \frac{1}{2}\iota _{\alpha \wedge \beta }\left[\pi ,\pi \right]=N\iota _{\alpha \wedge \beta }\Phi,\\
&\mathrm{(iii)}& N\iota _{\alpha \wedge \beta }\mathcal{L}_X\Phi -\iota _{\alpha \wedge \beta }\mathcal{L}_{NX}\Phi -\iota _{(\mathcal{L}_XN^*)(\alpha \wedge \beta )}\Phi =0
\end{eqnarray*}
for any $X$ in $\mathfrak{X}(M)$ and $\alpha$ and $\beta $ in $\Omega ^1(M)$, where $\iota _{\alpha \wedge \beta }:=\iota _\beta \iota _\alpha $ and $\iota _{(\mathcal{L}_XN^*)(\alpha \wedge \beta )}:=\iota _{(\mathcal{L}_XN^*)\alpha \wedge \beta }+\iota _{\alpha \wedge (\mathcal{L}_XN^*)\beta }$.

Furthermore, since quasi-Lie bialgebroids (of course, Lie bialgebroids also) construct Courant algebroids \cite{LWX}, \cite{R}, \cite{R2}, we can obtain a new Courant algebroid structure on $TM\oplus T^*M$ from a pseudo-Poisson Nijenhuis structure on $M$ similar to a Poisson Nijenhuis and a Poisson quasi-Nijenhuis structure on $M$. In other words, we can say that a pseudo-Poisson Nijenhuis structure is a new material for constructing a Courant algebroid structure on $TM\oplus T^*M$. Therefore a pseudo-Poisson Nijenhuis structure on $M$ complements the bottom left of the correspondence table below:

\begin{center}
\begin{tabular}{IcIc|c|} \hline
\multicolumn{3}{|c|}{a Courant algebroid structure \cite{LWX} on $TM\oplus T^*M$}\\ \hline
\multicolumn{1}{|c|}{a quasi-Lie bialgebroid \cite{R}}    & a Lie bialgebroid \cite{MX}       & a quasi-Lie bialgebroid \cite{R}\\
\multicolumn{1}{|c|}{$((TM)_N, d_\pi ,\Phi )$}    & $((TM)_N,(T^*M)_\pi )$   & $((T^*M)_\pi , d_N ,\phi )$ \\ \wcline{1-1}\cline{2-3}
a pseudo-Poisson Nijenhuis  & a Poisson Nijenhuis \cite{MM}     & a Poisson quasi-Nijenhuis \cite{SX}\\
$(\pi ,N, \Phi )$           & $(\pi ,N)$               & $(\pi ,N,\phi )$ \\
$\pi :$ a $2$-vector field  & $\pi :$ a Poisson        & $\pi :$ a Poisson \\
$N:$ a Nijenhuis            & $N:$ a Nijenhuis         & $N:$ a bundle map \\
$\Phi :$ a $3$-vector field &                          & $\phi :$ a $3$-form \\ \wcline{1-1}\cline{2-3}
\end{tabular}
\end{center} 

All of the pairs $(\pi ,N)$ of the bottom of the correspondence table above are ``compatible''. The condition that a $2$-vector field $\pi $ and a $(1,1)$-tensor $N$ on $M$ are compatible is very important in studying Poisson Nijenhuis, pseudo-Poisson Nijenhuis and Poisson quasi-Nijenhuis manifolds. In this paper, we generalize several properties related to the compatibility so that they can be used with as few assumptions as possible. For example, Poisson Nijenhuis hierarchy \cite{MM2}, \cite{KMa} and a relation with a brackets on the tangent and the cotangent bunble \cite{SX}, \cite{K} and so on.


Under the assumption that a $2$-vector field $\pi $ is nondegenerate, we can reduce one of the conditions for a triple $(\pi ,N,\Phi )$ to be a pseudo-Poisson Nijenhuis structure. In this case, since there is a unique nondegenerate $2$-form $\omega $ corresponding to $\pi $, we can rewrite the definition of pseudo-Poisson Nijenhuis structures by words of the differential forms. Therefore we obtain the definition of pseudo-stmplectic Nijenhuis structures as an equivalent structures to pseudo-Poisson Nijenhuis structures of which the $2$-vector field is nondegenerate:

\begin{defn}
Let $M$ be a $C^\infty $-manifold, $\omega $ a nondegenerate $2$-form on $M$, a $(1,1)$-tensor $N$ a Nijenhuis structure on $M$ compatible with $\pi $ corresponding to $\omega $, and $\phi $ a closed $3$-form on $M$. Then a triple $(\omega , N, \phi )$ is a {\it pseudo-symplectic Nijenhuis structure} on $M$ if the following holds:
\begin{eqnarray*}
\iota _{X\wedge Y}d\omega=N^*\iota _{X\wedge Y}\phi\ (X,Y \in \mathfrak{X}(M)).
\end{eqnarray*}
\end{defn}
\noindent
Moreover we show that pseudo-symplectic Nijenhuis structures $(\omega  ,N,\phi )$ induce twisted Poisson structures $(\pi _N, \phi )$ \cite{SW}. The property can be considered to be a generalization of the first step of the hierarchy of a Poisson Nijenhuis structure since a pair $(\pi _N, N)$ is compatible.

This paper is constructed as follows. We recall the definitions of Courant algebroids and quasi-Lie bialgebroids in section \ref{Preliminaries}. In section \ref{Compatible pairs}, we expand a general theory of the compatibility of a $2$-vector field and a $(1,1)$-tensor. This also plays an important role to study Poisson Nijenhuis, pseudo-Poisson Nijenhuis and Poisson quasi-Nijenhuis structures uniformly. In section \ref{Pseudo-Poisson Nijenhuis manifolds}, we define pseudo-Poisson Nijenhuis manifolds and show that there is a one-to-one correspondence between a pseudo-Poisson Nijenhuis manifold $(M, \pi ,N,\Phi )$ and a quasi-Lie bialgebroid $((TM)_N, d_\pi ,\Phi )$, which is the main theorem in this paper. It is the contents of section \ref{Pseudo-symplectic Nijenhuis manifolds} to define pseudo-symplectic Nijenhuis structures, 
and to investigate properties of those. In particular, we show that a pseudo-symplectic Nijenhuis structure induces a twisted Poisson structure \cite{SW}.

 \section{Preliminaries}\label{Preliminaries}
 We begin with recalling the definitions of Courant algebroids.
 
\begin{defn}[\cite{LWX}]
A {\it Courant algebroid} is a vector bundle $E\longrightarrow M$ equipped with a nondegenerate symmetric bilinear form $\langle \cdot ,\cdot \rangle $ (called the {\it pairing}) on the bundle, a skew-symmetric bracket $[\![\cdot ,\cdot ]\!]$ on $\Gamma (E)$ and a bundle map $\rho :E\longrightarrow TM$ such that the following properties are satisfied: for any $e, e_1, e_2, e_3$ in $\Gamma(E)$, any $f$ and $g$ in $C^\infty (M)$, 
\begin{enumerate}
\item[{\rm (i)}] $\sum_{\mathrm{Cycl}(e_1,e_2,e_3)}^{}[\![[\![e_1,e_2]\!],e_3]\!]=\frac{1}{3}\sum_{\mathrm{Cycl}(e_1,e_2,e_3)}^{}\mathcal{D}\langle [\![e_1,e_2]\!],e_3\rangle ;$
\item[{\rm (ii)}] $\rho ([\![e_1,e_2]\!])=[\rho(e_1), \rho(e_2)];$
\item[{\rm (iii)}] $[\![e_1,fe_2]\!]=f[\![e_1,e_2]\!]+(\rho(e_1)f)e_2-\langle e_1,e_2 \rangle \mathcal{D}f;$
\item[{\rm (iv)}] $\rho \circ \mathcal{D}=0,$ {\rm i.e.,} $\langle \mathcal{D}f,\mathcal{D}g\rangle =0 ;$
\item[{\rm (v)}] $\rho (e)\langle e_1,e_2\rangle =\langle [\![e,e_1]\!]+\mathcal{D}\langle e,e_1\rangle ,e_2\rangle +\langle e_1, [\![e,e_2]\!]+\mathcal{D}\langle e,e_2\rangle \rangle,$
\end{enumerate}
where $\mathcal{D}:C^\infty (M)\longrightarrow \Gamma(E)$ is the smooth map defined by\\
$$\langle \mathcal{D}f, e\rangle =\frac{1}{2}\rho (e)f.$$
The map $\rho $ and the operator $[\![\cdot , \cdot ]\!]$ are called an {\it anchor map} and a {\it Courant bracket}, respectively.
\end{defn}

A Courant algebroid is not a Lie algebroid since the Jacobi identity is not satisfied due to (i). The following example is fundamental.

\begin{example}[\cite{LWX}] \label{std-C}
The direct sum $TM\oplus T^*M$ on a $C^\infty $-manifold $M$ is a Courant algebroid. Here the anchor map $\rho $, the pairing $\langle \cdot ,\cdot \rangle $ and the Courant bracket $[\![\cdot ,\cdot ]\!]$ are given by
\begin{eqnarray}
&\rho (X+\xi )=X,\\
&\langle X+\xi , Y+\eta \rangle =\frac{1}{2}(<\xi ,Y>+<\eta ,X>), \label{pairing}\\
&[\![X+\xi ,Y+\eta ]\!]=[X,Y]+\mathcal{L}_X\eta -\mathcal{L}_Y\xi +\frac{1}{2}d(<\xi ,Y>-<\eta ,X>),
\end{eqnarray}
where $X$ and $Y$ are in $\mathfrak{X}(M)$, and $\xi $ and $\eta $ are in $\Omega ^1(M)$. This is called {\it the standard Courant algebroid}.
\end{example}

Next we shall recall the definition of quasi-Lie bialgebroids.

\begin{defn}[\cite{R}]
A {\it quasi-Lie bialgebroid} is a Lie algebroid $(A, [\cdot ,\cdot ]_A, a)$ equipped with a degree-one derivation $d_*$ of the Gerstenhaber algebra $(\Gamma (\Lambda ^*A), \wedge ,[\cdot ,\cdot]_A)$ and a 3-section of $A$, $\phi _A$ in $\Gamma (\Lambda ^3A)$ such that
\begin{eqnarray}
&d_*^2=[\phi _A, \cdot ]_A,\\
&d_*\phi _A=0.
\end{eqnarray}
\end{defn}

If the 3-section $\phi _A$ is equal to $0$, the quasi-Lie bialgebroid $(A, d_*,\phi _A)$ is just a Lie bialgebroid $(A, d_*)$.

\begin{example}[\cite{R}, \cite{R2}]\label{qLbia-to-Ca}
Let $(A, d_*,\phi _A)$ be a quasi-Lie bialgebroid, where $A=(A, [\cdot ,\cdot]_A, a)$, and $d_A: \Gamma (\Lambda ^*A^*)\rightarrow \Gamma (\Lambda ^{*+1}A^*)$ be the Lie algebroid derivative of $A$. Its double $E=A\oplus A^*$ has naturally a Courant algebroid structure. Namely, it is equipped with an anchor map $\rho $, a pairing $\langle \cdot ,\cdot \rangle $ and a Courant bracket $[\![\cdot , \cdot ]\!]$ given by the following: for any $X,Y$ in $\Gamma (A)$, any $\xi $ and $\eta $ in $\Gamma (A^*)$,
\begin{eqnarray*}
&\rho (X+\xi )=a(X)+a_*(\xi ),\\
&\langle X+\xi , Y+\eta \rangle =\frac{1}{2}(<\xi ,Y>+<\eta ,X>),\\
&[\![X,Y]\!]=[X,Y]_A\\
&[\![\xi ,\eta ]\!]=[\xi ,\eta ]_{A^*}+\phi _A(X,Y,\cdot )\\
&[\![X,\xi ]\!]=(\iota _Xd_A\xi +\frac{1}{2}d_A<\xi ,X>)\\
&\quad \quad \quad \quad \quad \quad -(\iota _\xi d_*X +\frac{1}{2}d_*<\xi ,X>),
\end{eqnarray*}
where the map $a_*:A^*\longrightarrow TM$ and the bracket $[\cdot ,\cdot ]_{A^*}$ are defined by
\begin{eqnarray*}
&a_*(\xi )f:=<\xi ,d_*f>,\\
&<[\xi, \eta ]_{A^*},X>:=a_*(\xi )<\eta ,Y>-a_*(\eta )<\xi ,X>-(d_*X)(\xi ,\eta ),
\end{eqnarray*}
respectively and $\iota _X$ and $\iota _\xi $ are the interior products defined by $\iota _X\zeta :=\zeta (X, \dots )$ and $\iota _\xi D:=D(\xi, \dots)$, respectively for any $\xi , \eta$ in $\Gamma (A^*)$, $X,Y$ in $\Gamma(A)$, $\zeta $ in $\Gamma (\Lambda ^*A^*)$, $D$ in $\Gamma(\Lambda ^*A)$ and $f$ in $C^\infty (M)$.

Taking $\phi _A=0$, we obtain the Courant algebroid structure of a double of a Lie bialgebroid in \cite{LWX}.
\end{example}

 \section{Compatible pairs}\label{Compatible pairs}

In this section, we consider the compatibility of a $2$-vector field and a $(1,1)$-tensor on a $C^\infty $-manifold, which plays an important role to define a Poisson Nijenhuis and a pseudo-Poisson Nijenhuis manifold. For that reason, first we begin with the definitions and properties of brackets defined by a $2$-vector field and a $(1,1)$-tensor. We generalize several properties of a Poisson Nijenhuis structure to that of a compatible pair of a $2$-vector field and a $(1,1)$-tensor. 
Moreover we show that the brackets gives a characterization of the compatibility of a $2$-vector field and a $(1,1)$-tensor, which is the main theorem of this section.

Let $M$ be a $C^\infty $-manifold, $\pi $ a $2$-vector field and $N$ a $(1,1)$-tensor. We define, for any $\alpha ,\beta $ in $\Omega ^1(M)$ and $X,Y$ in $\mathfrak{X}(M),$
\begin{eqnarray}
&[\alpha ,\beta ]_\pi :=\mathcal{L}_{\pi ^\sharp \alpha }\beta -\mathcal{L}_{\pi ^\sharp \beta }\alpha -d<\pi ^\sharp \alpha ,\beta >,\\
&[X,Y]_N:=[NX,Y]+[X,NY]-N[X,Y],
\end{eqnarray}
where $\pi ^\sharp :T^*M\longrightarrow TM$ is the bundle map over $M$ defined by $<\pi ^\sharp \alpha ,\beta>:=\pi (\alpha ,\beta )$. It is easy to see that these brackets are bilinear and anti-symmetry. Moreover these satisfy the Leibniz rule, i.e., for any $f$ in $C^\infty (M), \alpha ,\beta $ in $\Omega ^1(M)$ and $X,Y$ in $\mathfrak{X}(M),$
\begin{eqnarray}
&[\alpha ,f\beta ]_\pi =((\pi^\sharp \alpha )f)Y+f[\alpha ,\beta ]_\pi ,\label{pi_Leibniz}\\
&[X,fY]_N=((NX)f)Y+f[X,Y]_N.\label{N_Leibniz}
\end{eqnarray}
From this, we obtain the derivation $d_\pi :\mathfrak{X}^*(M)\longrightarrow \mathfrak{X}^{*+1}(M)$ defined by
\begin{eqnarray}\label{gaibibun}
(d_{\pi } D)(\alpha _0, \dots , \alpha _k)&=&\sum_{i=0}^k(-1)^i\pi ^\sharp (\alpha _i)(D(\alpha _0,\dots ,\hat{\alpha _i},\dots ,\alpha _k))\nonumber \\
&&\!\!\!\!+\sum _{i<j}^{}(-1)^{i+j}D([\alpha _i,\alpha _j]_\pi , \alpha _0,\dots ,\hat{\alpha _i},\dots ,\hat{\alpha _j},\dots ,\alpha _k),\nonumber \\
\end{eqnarray}
where $D$ is in $\mathfrak{X}^k(M)$ and $\alpha _i$'s are in $\Omega ^1(M)$. By replacing $\pi ^\sharp $ and $[\cdot ,\cdot ]_\pi $ with $N$ and $[\cdot ,\cdot ]_N$ respectively, the derivation $d_N:\Omega ^*(M)\longrightarrow \Omega ^{*+1}(M)$ is also defined similarly. Then for any $D$ in $\mathfrak{X}^k(M)$, it follows that $d_\pi D=[\pi , D]$. Furthermore the Lie derivative $\mathcal{L}_\alpha ^\pi$ and $\mathcal{L}_X^N$ are defined by the Cartan formula
\begin{eqnarray}
\mathcal{L}_\alpha ^\pi:=d_\pi \iota _\alpha +\iota _\alpha d_\pi ,\ \ \ \ \ \mathcal{L}_X^N:=d_N \iota _X +\iota _X d_N 
\end{eqnarray}
for any $\alpha $ in $\Omega ^1(M)$ and $X$ in $\mathfrak{X}(M)$ and are extended on $\mathfrak{X}^*(M)$ and $\Omega ^*(M)$ the same as the usual Lie derivative $\mathcal{L}_X$ respectively. Then it follows that
$$\mathcal{L}_\alpha ^\pi \beta =[\alpha ,\beta ]_\pi ,\ \ \ \ \ \mathcal{L}_X^NY=[X,Y]_N.$$

\begin{rem}\label{Lie algebroid}
The above brackets are {\it not} Lie brackets in general. The bracket $[\cdot ,\cdot ]_\pi $ is a Lie bracket on $\Omega ^1(M)$ if and only if the 2-vector field $\pi $ on $M$ is a {\it Poisson structure}, i.e., $[\pi ,\pi ]=0$. Then the cotangent bundle $(T^*M)_\pi :=(T^*M,[\cdot ,\cdot ]_\pi ,\pi ^\sharp )$ is a Lie algebroid. The bracket $[\cdot ,\cdot ]_N$ is a Lie bracket on $\mathfrak{X}(M)$ if and only if $N$ is a {\it Nijenhuis structure} on $M$, i.e., the Nijenhuis torsion
$$\mathcal{T}_N(X,Y):=[NX,NY]-N[X,Y]_N$$
vanishes for any $X$ and $Y$ in $\mathfrak{X}(M)$. Then the tangent bundle $(TM)_N=(TM,[\cdot ,\cdot ]_N,N)$ is a Lie algebroid.
\end{rem}

By observing carefully the ploof of the existence and uniqueness theorem of the Schouten bracket of the usual bracket for vector fields (for example, see \cite{MR}), we can show that a similar one is also constructed in the following situation:

\begin{theorem}\label{Schouten}
Let $(A,a)$ be an {\it anchored vector bundle} over $M$, i.e., $a:A\longrightarrow TM$ is a bundle map over $M$, and $[\cdot ,\cdot ]_A$ a anti-symmetric bilinear bracket on $\Gamma(A)$ satisfying
\begin{eqnarray}
[X,fY]_A=(a(X)f)Y+f[X,Y]_A
\end{eqnarray}
for any $X, Y$ in $\Gamma(A)$ and $f$ in $C^\infty (M)$. Then there is a unique bilinear operator $[\cdot ,\cdot ]_A:\Gamma(\Lambda ^*A)\times \Gamma(\Lambda ^*A)\longrightarrow \Gamma(\Lambda ^*A)$, called {\it the generalized Schouten bracket} or simply {\it the Schouten bracket}, that satisfies the following properties: 
\begin{enumerate}
\item[(i)] It is a biderivation of degree $-1$, that is, it is bilinear,
\begin{eqnarray}
\mbox{deg}[D_1,D_2]_A=\mbox{deg}D_1+\mbox{deg}D_2-1,
\end{eqnarray}
and
\begin{eqnarray}
[D_1,D_2\wedge D_3]_A\!\!\!&=&\![D_1,D_2]_A\wedge D_3\\
&&\ \ +(-1)^{\left(\mbox{deg}D_1+1\right)\mbox{deg}D_2}D_2\wedge [D_1, D_3]_A.
\end{eqnarray}
for $D_i$ in $\Gamma (\Lambda ^*A)$,
\item[(ii)] It is determined on $C^\infty (M)$ and $\Gamma(A)$ by
\begin{enumerate}
\item[(a)]$[f,g]_A=0\ (f,g \in C^\infty(M));$
\item[(b)]$[X,f]_A=a(X)f\ (f \in C^\infty (M), X\in \Gamma (A));$
\item[(c)]$[X,Y]_A\ (X,Y\in \Gamma(A))$ is the original bracket on $\Gamma(A)$.
\end{enumerate}
\item[(iii)] $[D_1,D_2]_A=(-1)^{\mbox{deg}D_1\mbox{deg}D_2}[D_2,D_1]_A$.
\end{enumerate}
\end{theorem}

\begin{rem}
In general, the Schouten bracket of a bracket $[\cdot ,\cdot ]_A$ on $\Gamma (A)$ does not satisfy the graded Jacobi identity because $[\cdot ,\cdot ]_A$ does not satisfy the Jacobi identity.
\end{rem}

Since $(TM,N)$ and $(T^*M,\pi^\sharp )$ are anchored vector bundles over $M$ and brackets $[\cdot ,\cdot ]_\pi $ and $[\cdot ,\cdot ]_N$ satisfy (\ref{pi_Leibniz}) and (\ref{N_Leibniz}) respectively, by Theorem \ref{Schouten}, $[\cdot ,\cdot ]_\pi $ and $[\cdot ,\cdot ]_N$ are extended to the Schouten bracket on $\Omega ^*(M)$ and on $\mathfrak{X}^*(M)$ respectively.

We define the concept related to a $2$-vector field and a $(1,1)$-tensor, called {\it the compatibility} of those.

\begin{defn}[\cite{MM2}, \cite{KMa}]\label{compatible}
The $2$-vector field $\pi $ on $M$ and the $(1,1)$-tensor $N$ on $M$ are {\it compatible} if those satisfy
\begin{eqnarray}
N\circ \pi ^\sharp =\pi ^\sharp \circ N^*, \label{kakan}
\end{eqnarray}
and the $(2,1)$-tensor
\begin{eqnarray}
C_\pi ^N(\alpha , \beta ):=[\alpha ,\beta ]_{N\pi ^\sharp }-[\alpha ,\beta ]_\pi ^{N^*}
\end{eqnarray}
vanishes, where for any $\alpha $ and $\beta $ in $\Omega ^1(M)$,
\begin{eqnarray}
&[\alpha ,\beta ]_{N\pi ^\sharp }:=\mathcal{L}_{N\pi ^\sharp \alpha }\beta -\mathcal{L}_{N\pi ^\sharp \beta }\alpha -d<N\pi ^\sharp \alpha ,\beta >\ \ \mbox{and}\\
&[\alpha ,\beta ]_\pi ^{N^*}:=[N^*\alpha , \beta ]_\pi +[\alpha ,N^*\beta ]_\pi -N^*[\alpha ,\beta ]_\pi .
\end{eqnarray}
A compatible pair $(\pi ,N)$ is a {\it Poisson Nijenhuis structure} if $\pi $ is Poisson and $N$ is Nijenhuis.
\end{defn}

Let $(\pi ,N)$ be a compatible pair and set $\pi _N(\alpha ,\beta ):=<N\pi ^\sharp \alpha ,\beta >$. Then it follows from (\ref{kakan}) that $\pi _N$ is a $2$-vector field on $M$. Hence under the assumption (\ref{kakan}), the bracket $[\cdot ,\cdot ]_{N\pi ^\sharp }$ can be rewritten as $[\cdot ,\cdot ]_{\pi _N}$. If $(\pi ,N)$ is a Poisson Nijenhuis structure on $M$, then $\pi _N$ is Poisson . 

For any compatible pair $(\pi ,N)$, we set $\pi _0:=\pi $ and define a $2$-vector field $\pi _{k+1}$ by the condition $\pi _{k+1}^\sharp  =N\circ \pi _k^\sharp $ inductively. In the case of a compatible pair $(\pi ,N)$ of which $N$ is Nijenhuis, the following proposition corresponding to the existence theorem of the hierarchy of Poisson Nijenhuis structures \cite{MM2}, \cite{KMa} can be shown in the same way as the theorem.

\begin{prop}\label{hierarchy}
Let $(\pi ,N)$ be a compatible pair on $M$ such that $N$ is Nijenhuis. Then all pairs $(\pi _k,N^p)$ are compatible pairs on $M$ such that $N^p$ are Nijenhuis. Furthermore for any $k,l\geq 0$ and $Q$ in $\mathfrak{X}^*(M)$, $[\pi _k,Q]_{N^{l+1}}=[\pi _{k+1},Q]_{N^l}$.
\end{prop}

The compatibility of a $2$-vector field $\pi $ and a $(1,1)$-tensor $N$ is equivalent to the following equations using the Schouten brackets of $[\cdot ,\cdot ]_\pi $ and $[\cdot ,\cdot]_N$. 

\begin{theorem}\label{compatible_equivalent}
Let $M$ be a $C^\infty $-manifold, $\pi $ a $2$-vector field on $M$ and $N$ a $(1,1)$-tensor on $M$. Then the following properties are equivalent:
\begin{enumerate}
\item[(i)]$\pi $ and $N$ are compatible; 
\item[(ii)]the operator $d_N$ is a derivation of the Schouten bracket $[\cdot ,\cdot ]_\pi:$
\begin{eqnarray}\label{N_derivation}
d_N[\xi _1 , \xi_2 ]_\pi =[d_N\xi _1, \xi _2 ]_\pi +(-1)^{\mathrm{deg}\xi _1+1}[\xi _1, d_N\xi_2 ]_\pi;
\end{eqnarray}
\item[(iii)]the operator $d_\pi$ is a derivation of the Schouten bracket $[\cdot ,\cdot ]_N:$
\begin{eqnarray}\label{pi_derivation}
d_\pi[D_1,D_2]_N=[d_\pi D_1, D_2]_N+(-1)^{\mathrm{deg}D_1+1}[D_1, d_\pi D_2]_N,
\end{eqnarray}
\end{enumerate}
where $\xi _i$'s are in $\Omega^*(M)$ and $D_i$'s are in $\mathfrak{X}^*(M)$.
\end{theorem}

In the case of that $\pi $ is Poisson, Theorem \ref{compatible_equivalent} coincides with Lemma 3.6 in \cite{SX} and Proposition 3.2 in \cite{K}. However to prove Proposition 3.2 in \cite{K}, properties for a Lie bialgebroid \cite{LWX} were used since $((TM)_N,(T^*M)_\pi )$ is a Lie bialgebroid, and Lemma 3.6 in \cite{SX} does not mention the equivalence of (i) and (iii) in Theorem \ref{compatible_equivalent}. Therefore Theorem \ref{compatible_equivalent} is worthy in the sense that these equivalence is indicated by eliminating conditions that do not require it. To prove Theorem \ref{compatible_equivalent}, we need the following lemma.

\begin{lemma}\label{lemma_compatible}
Let $\pi $ be a $2$-vector field on $M$ and $N$ a $(1,1)$-tensor on $M$. Assume that $\pi $ and $N$ satisfy the condition (\ref{kakan}). Then the pair $(\pi ,N)$ is compatible if and only if for any $f$ in $C^\infty (M)$ and $X$ in $\mathfrak{X}(M)$,
\begin{eqnarray}\label{lemma_condition}
\mathcal{L}_{d_Nf}^\pi X=-[d_\pi f,X]_N.
\end{eqnarray}
\end{lemma}

\begin{proof}
For any $\xi $ in $\Omega ^1(M)$, we calculate
\begin{eqnarray*}\label{lemma_1}
<\mathcal{L}_{d_Nf}^\pi X, \xi >&=&\mathcal{L}_{d_Nf}^\pi<X,\xi >-<X,\mathcal{L}_{d_Nf}^\pi \xi>\\
                                &=&(\pi ^\sharp N^*df)<X,\xi >-<X,[N^*df, \xi]_\pi>\\
                                &=&(\pi _N^\sharp df)<X,\xi >-<X,[df, \xi]_\pi^{N^*}>\\
                                &\quad & \quad +<X, [df, N^*\xi ]_\pi >-<X,N^*[df,\xi]_\pi >\\
                                &=&(\pi _N^\sharp df)<X,\xi >-<X,[df, \xi]_\pi^{N^*}>\\
                                &\quad & \quad +<X, \mathcal{L}_{\pi ^\sharp df}(N^*\xi)>-<NX,\mathcal{L}_{\pi ^\sharp df}\xi >\\
                                &=&(\pi _N^\sharp df)<X,\xi >-<X,[df, \xi]_\pi^{N^*}>\\
                                &\quad & \quad +<N[d_\pi f,X], \xi>-<[d_\pi f,NX],\xi >.
\end{eqnarray*}
On the other hand, we obtain
\begin{eqnarray*}\label{lemma_2}
<[d_\pi f,X]_N,\xi >&=&<[Nd_\pi f,X]+[d_\pi f,NX]-N[d_\pi f,X],\xi >\\
                    &=&<[d_{\pi _N}f,X],\xi >+<[d_\pi f,NX]-N[d_\pi f,X],\xi >\\
                    &=&<[[\pi _N,f],X],\xi >+<[d_\pi f,NX]-N[d_\pi f,X],\xi >\\
                    &=&<-[[X,\pi _N],f]-[\pi _N, [f,X]],\xi >\\
                    &\quad &\quad \quad \quad \quad \quad +<[d_\pi f,NX]-N[d_\pi f,X],\xi >\\
                    &=&-<[d_{\pi _N}X,f],\xi >+<[\pi _N, Xf],\xi >\\
                    &\quad &\quad \quad \quad \quad \quad +<[d_\pi f,NX]-N[d_\pi f,X],\xi >\\
                    &=&-(d_{\pi _N}X)(df,\xi )+\pi _N(d(Xf),\xi )\\
                    &\quad &\quad \quad \quad \quad \quad +<[d_\pi f,NX]-N[d_\pi f,X],\xi >\\
                    &=&-(\pi _N^\sharp df)<X,\xi >+(\pi _N^\sharp \xi)<X,df>\\
                    &\quad &\quad \quad \quad +<X,[df,\xi ]_{\pi _N}>-(\pi _N^\sharp \xi )(Xf)\\
                    &\quad &\quad \quad \quad \quad +<[d_\pi f,NX]-N[d_\pi f,X],\xi >\\
                    &=&-(\pi _N^\sharp df)<X,\xi >+<X,[df,\xi ]_{\pi _N}>\\
                    &\quad &\quad \quad \quad \quad +<[d_\pi f,NX],\xi>-<N[d_\pi f,X],\xi >.
\end{eqnarray*}
Therefore we find
\begin{eqnarray*}\label{lemma_3}
<\mathcal{L}_{d_Nf}^\pi X+[d_\pi f,X]_N,\xi >&=&<X,[df,\xi ]_{\pi _N}-[df,\xi ]_\pi ^{N^*}>\\
                                             &=&<X,C_\pi ^N(df,\xi )>.
\end{eqnarray*}
Because the exact $1$-forms generate locally the $1$-forms as a $C^\infty (M)$-module and $C_\pi ^N$ is tensorial, we obtain the equivalence to prove.
\end{proof}

\begin{proof}[Proof of Theorem \ref{compatible_equivalent}]
The equivalence of (i) and (ii) can be proved similarly as Proposition 3.2 in \cite{K}. 
We shall prove the equivalence of (i) and (iii). We set for any $D_1$ and $D_2$ in $\mathfrak{X} ^*(M)$,
\begin{eqnarray}\label{A_pi}
A_{\pi ,N}(D_1,D_2):=d_\pi [D _1 , D_2 ]_N -[d_\pi D_1, D_2 ]_N-(-1)^{\mathrm{deg}D_1+1}[D_1, d_\pi D_2 ]_N. 
\end{eqnarray}
Then by straightforward calculation, for any $f,g$ in $C^\infty (M)$, $X,Y$ in $\mathfrak{X}^1(M)$, $D_1,D_2$ and $D_3$ in $\mathfrak{X}^*(M)$, we obtain
\begin{eqnarray*}
         A_{\pi ,N}(f,g)&=&<(N\pi ^\sharp -\pi ^\sharp N^*)df,dg>,\\
    <A_{\pi ,N}(X,f),dg>&=&((N\pi ^\sharp -\pi ^\sharp N^*)df)<X,dg>\\
                        &\quad &\quad \quad -<X,d<(N\pi ^\sharp -\pi ^\sharp N^*)df,dg>>\\
                        &\quad &\quad \quad \quad \quad \quad \quad \quad \quad \quad \quad +<X,C_N^\pi(df,dg)>,\\
(A_{\pi ,N}(X,Y))(df,dg)&+&(d(C_N^\pi(df,dg)))(X,Y)\\
                        &=&+<df, \mathcal{L}_{d_N<Y,dg>}^\pi X+[d_\pi <Y,dg>,X]_N>\\
                        &\quad &-<df, \mathcal{L}_{d_N<X,dg>}^\pi Y+[d_\pi <X,dg>,Y]_N>\\
                        &\quad &+<dg, \mathcal{L}_{d_N<X,df>}^\pi Y+[d_\pi <X,df>,Y]_N>\\
                        &\quad &-<dg, \mathcal{L}_{d_N<Y,df>}^\pi X+[d_\pi <Y,df>,X]_N>\\
                        &\quad &+<(\pi ^\sharp N^*-N\pi ^\sharp )d<X,dg>,d<Y,df>>\\
                        &\quad &-<(\pi ^\sharp N^*-N\pi ^\sharp )d<X,df>,d<Y,dg>>,\\
A_{\pi ,N}(D_1,D_2\wedge D_3 )&=&A_{\pi ,N}(D_1,D_2) \wedge D_3\\
                              &\quad &\quad \quad +(-1)^{\mbox{deg}D_1}D_2\wedge A_{\pi ,N}(D_1,D_3),\\
A_{\pi ,N}(D_1,D_2)&=&(-1)^{(\mbox{deg}D_1-1)(\mbox{deg}D_2-1)}A_{N,\pi }(D_2,D_1),                        
\end{eqnarray*}
so that the conclusion follows from these equations and Lemma \ref{lemma_compatible}. 
\end{proof}

\section{Pseudo-Poisson Nijenhuis manifolds}\label{Pseudo-Poisson Nijenhuis manifolds}

In this section, we define Pseudo-Poisson Nijenhuis manifolds and investigate properties of them. 

\begin{defn}\label{pPN}
Let $M$ be a $C^\infty $-manifold, $\pi $ a $2$-vector field on $M$, a $(1,1)$-tensor $N$ on $M$ a Nijenhuis structure compatible with $\pi $, and $\Phi $ a $3$-vector field on $M$. Then a triple $(\pi, N, \Phi )$ is a {\it pseudo-Poisson Nijenhuis structure} on $M$ if the following conditions hold:
\begin{eqnarray}
&\mathrm{(i)}& [\pi ,\Phi ]=0,\label{(i)}\\
&\mathrm{(ii)}& \frac{1}{2}\iota _{\alpha \wedge \beta }\left[\pi ,\pi \right]=N\iota _{\alpha \wedge \beta }\Phi,\label{(ii)}\\
&\mathrm{(iii)}& N\iota _{\alpha \wedge \beta }\mathcal{L}_X\Phi -\iota _{\alpha \wedge \beta }\mathcal{L}_{NX}\Phi -\iota _{(\mathcal{L}_XN^*)(\alpha \wedge \beta )}\Phi =0,\label{(iii)}
\end{eqnarray}
for any $X$ in $\mathfrak{X}(M)$, $\alpha$ and $\beta $ in $\Omega ^1(M)$, where $\iota _{\alpha \wedge \beta }:=\iota _\beta \iota _\alpha $ and $\iota _{(\mathcal{L}_XN^*)(\alpha \wedge \beta )}:=\iota _{(\mathcal{L}_XN^*)\alpha \wedge \beta }+\iota _{\alpha \wedge (\mathcal{L}_XN^*)\beta }$. The quadruple $(M, \pi, N, \Phi )$ is called a {\it pseudo-Poisson Nijenhuis manifold}.
\end{defn}

\begin{rem}
The reason why we use not ``quasi-'' but ``pseudo-'' is to avoid confusion with another notion quasi-Poisson manifold in \cite{AK}, \cite{AKMe}.
\end{rem}

Now we describe the main theorem in this section. This is one of the fundamental properties of pseudo-Poisson Nijenhuis manifolds. A similar result for Poisson quasi-Nijenhuis manifolds is also known \cite{SX}.

\begin{theorem}\label{pPN-to-qLbia}
Let $M$ be a $C^\infty $-manifold, $\pi $ a $2$-vector field on $M$, $N$ a Nijenhuis structure on $M$ compatible with $\pi $ and $\Phi $ a $3$-vector field on $M$. Then a quadruple $(M, \pi, N, \Phi )$ is a pseudo-Poisson Nijenhuis manifold if and only if $((TM)_N, d_\pi , \Phi)$ is a quasi-Lie bialgebroid.
\end{theorem}

\proof
Since a $(1,1)$-tensor $N$ is Nijenhuis, the Lie algebroid $(TM)_N$ is well-defined. A triple $((TM)_N, d_\pi , \Phi)$ is a quasi-Lie bialgebroid if and only if the following three conditions hold: i) $d_\pi $ is a degree-one derivation of the Gerstenhaber algebra $(\mathfrak{X}^*(M), \wedge , [\cdot ,\cdot ]_N)$, ii) $d_\pi ^2=[\Phi ,\cdot ]_N$ and iii) $d_\pi \Phi =0$ by the definition.

i) means that (\ref{pi_derivation}) 
holds. This condition is equivalent to the compatibility of $\pi $ and $N$ by Theorem \ref{compatible_equivalent}.

Next, For any $f$ in $C^\infty (M)$, any $\alpha $ and $\beta $ in $\Omega ^1(M)$, we compute
\begin{eqnarray*}
(d_\pi ^2f)(\alpha , \beta )&=&[\pi , [\pi , f]](\alpha , \beta )=\frac{1}{2}[[\pi , \pi ], f](\alpha , \beta )\\
                            &=&\frac{1}{2}\iota _{df}[\pi ,\pi ](\alpha , \beta )=\frac{1}{2}[\pi ,\pi ](df, \alpha , \beta )\\
                            &=&\frac{1}{2}[\pi ,\pi ](\alpha , \beta , df)=\frac{1}{2}\iota _{\alpha \wedge \beta}[\pi ,\pi ](df),
\end{eqnarray*}
where the second equality follows from the graded Jacobi identity of the Schouten bracket $[\cdot ,\cdot ]$, and the fact is used that $[D,f]=(-1)^{k+1}\iota _{df}D$ for any $D$ in $\mathfrak{X}^k (M)$ in the third equality. On the other hand, we have
\begin{eqnarray*}
[\Phi ,f]_N(\alpha , \beta )&=&\iota _{N^*df}\Phi (\alpha , \beta )=\Phi (N^*df, \alpha , \beta )\\
                            &=&\Phi (\alpha , \beta , N^*df)=\iota _{\alpha \wedge \beta}\Phi (N^*df)\\
                            &=&(N\iota _{\alpha \wedge \beta}\Phi )(df),
\end{eqnarray*}
where we use the fact that $[D,f]_N=(-1)^{k+1}\iota _{N^*df}D$ for any $D$ in $\mathfrak{X}^k (M)$ in the first step. Therefore it follows that $d_\pi ^2=[\Phi , \cdot ]_N$ on $C^\infty (M)$ if and only if the equality (\ref{(ii)}) holds as a linear map on the exact $1$-forms. By $C^\infty (M)$-linearity of (\ref{(ii)}) and the fact that the exact $1$-forms generate locally the $1$-forms as a $C^\infty (M)$-module, the equality (\ref{(ii)}) holds on $\Omega ^1(M)$ if and only if $d_\pi ^2=[\Phi , \cdot ]_N$ holds on $C^\infty (M)$.

Next, under the assumption that the equality (\ref{(ii)}) holds on $\Omega ^1(M)$, for any $X$ in $\mathfrak{X}(M)$, any $\alpha , \beta $ and $\gamma $ in $\Omega ^1(M)$, we obtain
\begin{eqnarray*}
(d_\pi ^2X)(\alpha , \beta , \gamma)&=&[\pi , [\pi ,X]](\alpha , \beta , \gamma)=\frac{1}{2}[[\pi ,\pi ],X](\alpha , \beta , \gamma)\\
                                    &=&-\frac{1}{2}[X, [\pi ,\pi ]](\alpha , \beta , \gamma)=-\frac{1}{2}(\mathcal{L}_X[\pi ,\pi ])(\alpha , \beta , \gamma)\\
                                    &=&-\frac{1}{2}\{\mathcal{L}_X([\pi ,\pi ](\alpha , \beta , \gamma))-[\pi ,\pi ](\mathcal{L}_X\alpha , \beta , \gamma)\\
                                    &\quad &\quad \quad \quad -[\pi ,\pi ](\alpha , \mathcal{L}_X\beta , \gamma)-[\pi ,\pi ](\alpha , \beta , \mathcal{L}_X\gamma)\}\\
                                    &=&-\mathcal{L}_X\left(\frac{1}{2}\iota _{\alpha \wedge \beta }[\pi ,\pi ](\gamma)\right)+\frac{1}{2}\iota _{\mathcal{L}_X\alpha \wedge \beta }[\pi ,\pi ](\gamma)\\
                                    &\quad &\quad \quad \quad +\frac{1}{2}\iota _{\alpha \wedge \mathcal{L}_X\beta }[\pi ,\pi ](\gamma)+\frac{1}{2}\iota _{\alpha \wedge \beta }[\pi ,\pi ](\mathcal{L}_X\gamma)\\
                                    &=&-\mathcal{L}_X\left((N\iota _{\alpha \wedge \beta }\Phi )(\gamma)\right)+(N\iota _{\mathcal{L}_X\alpha \wedge \beta }\Phi )(\gamma)\\
                                    &\quad &\quad \quad \quad +(N\iota _{\alpha \wedge \mathcal{L}_X\beta }\Phi )(\gamma)+(N\iota _{\alpha \wedge \beta }\Phi )(\mathcal{L}_X\gamma)\\
                                    &=&-\mathcal{L}_X\left(\iota _{\alpha \wedge \beta }\Phi (N^*\gamma)\right)+\iota _{\mathcal{L}_X\alpha \wedge \beta }\Phi (N^*\gamma)\\
                                    &\quad &\quad \quad \quad +\iota _{\alpha \wedge \mathcal{L}_X\beta }\Phi (N^*\gamma)+\iota _{\alpha \wedge \beta }\Phi (N^*\mathcal{L}_X\gamma)\\
                                    &=&-\mathcal{L}_X\left(\Phi (\alpha , \beta ,N^*\gamma)\right)+\Phi (\mathcal{L}_X\alpha , \beta ,N^*\gamma)\\
                                    &\quad &\quad \quad \quad +\Phi (\alpha ,\mathcal{L}_X\beta ,N^*\gamma)+\Phi (\alpha ,\beta ,N^*\mathcal{L}_X\gamma)\\
                                    &=&-\mathcal{L}_X\left(\Phi (\alpha , \beta ,N^*\gamma)\right)+\Phi (\mathcal{L}_X\alpha , \beta ,N^*\gamma)\\
                                    &\quad &\quad +\Phi (\alpha ,\mathcal{L}_X\beta ,N^*\gamma)+\Phi (\alpha ,\beta ,\mathcal{L}_X(N^*\gamma)-(\mathcal{L}_XN^*)\gamma)\\
                                    &=&-\mathcal{L}_X\left(\Phi (\alpha , \beta ,N^*\gamma)\right)+\Phi (\mathcal{L}_X\alpha , \beta ,N^*\gamma)\\
                                    &\quad &\quad +\Phi (\alpha ,\mathcal{L}_X\beta ,N^*\gamma)+\Phi (\alpha ,\beta ,\mathcal{L}_X(N^*\gamma))\\
                                    &\quad &\quad -\Phi (\alpha ,\beta ,(\mathcal{L}_XN^*)\gamma)\\
                                    &=&-\left(\mathcal{L}_X\Phi \right)(\alpha , \beta ,N^*\gamma)-\Phi (\alpha ,\beta ,(\mathcal{L}_XN^*)\gamma),
\end{eqnarray*}
where the second equality follows from the graded Jacobi identity of $[\cdot ,\cdot ]$, and we use the equality (\ref{(ii)}) in the seventh equality .  On the other hand, we obtain
\begin{eqnarray*}
[\Phi ,X]_N(\alpha , \beta ,\gamma)&=&-[X,\Phi ]_N(\alpha , \beta ,\gamma)=-(\mathcal{L}_X^N\Phi )(\alpha , \beta ,\gamma)\\
                                   &=&-\mathcal{L}_X^N(\Phi (\alpha , \beta ,\gamma))+\Phi (\mathcal{L}_X^N\alpha , \beta ,\gamma)\\
                                   &\quad &\quad \quad +\Phi (\alpha , \mathcal{L}_X^N\beta ,\gamma)+\Phi (\alpha , \beta ,\mathcal{L}_X^N\gamma)\\
                                   &=&-\mathcal{L}_{NX}(\Phi (\alpha , \beta ,\gamma))\\
                                   &\quad &+\Phi (\mathcal{L}_{NX}\alpha -(\mathcal{L}_XN^*)\alpha , \beta ,\gamma)\\
                                   &\quad &+\Phi (\alpha , \mathcal{L}_{NX}\beta -(\mathcal{L}_XN^*)\beta ,\gamma)\\
                                   &\quad &+\Phi (\alpha , \beta ,\mathcal{L}_{NX}\gamma -(\mathcal{L}_XN^*)\gamma)\\
                                   &=&-\mathcal{L}_{NX}(\Phi (\alpha , \beta ,\gamma))+\Phi (\mathcal{L}_{NX}\alpha , \beta ,\gamma)\\
                                   &\quad &\quad \quad +\Phi (\alpha , \mathcal{L}_{NX}\beta ,\gamma)+\Phi (\alpha , \beta ,\mathcal{L}_{NX}\gamma)\\
                                   &\quad &\quad \quad -\Phi ((\mathcal{L}_XN^*)\alpha , \beta ,\gamma)-\Phi (\alpha , (\mathcal{L}_XN^*)\beta ,\gamma)\\
                                   &\quad &\quad \quad -\Phi (\alpha , \beta ,(\mathcal{L}_XN^*)\gamma)\\
                                   &=&-(\mathcal{L}_{NX}\Phi )(\alpha , \beta ,\gamma)-\Phi ((\mathcal{L}_XN^*)\alpha , \beta ,\gamma)\\
                                   &\quad &\quad \quad -\Phi (\alpha , (\mathcal{L}_XN^*)\beta ,\gamma)-\Phi (\alpha , \beta ,(\mathcal{L}_XN^*)\gamma),
\end{eqnarray*}
where we use the property that $\mathcal{L}_X^N\alpha =\mathcal{L}_{NX}\alpha -(\mathcal{L}_XN^*)\alpha $ for any $X$ in $\mathfrak{X}(M)$ and any $\alpha $ in $\Omega ^1(M)$. Therefore, we obtain
\begin{eqnarray*}
(d_\pi ^2-[\Phi ,X]_N)(\alpha , \beta ,\gamma)&=&-\left(\mathcal{L}_X\Phi \right)(\alpha , \beta ,N^*\gamma)-\Phi (\alpha ,\beta ,(\mathcal{L}_XN^*)\gamma)\\
                                              &\quad &\quad +(\mathcal{L}_{NX}\Phi )(\alpha , \beta ,\gamma)+\Phi ((\mathcal{L}_XN^*)\alpha , \beta ,\gamma)\\
                                              &\quad &\quad +\Phi (\alpha , (\mathcal{L}_XN^*)\beta ,\gamma)+\Phi (\alpha , \beta ,(\mathcal{L}_XN^*)\gamma)\\
                                              &=&-\left(\mathcal{L}_X\Phi \right)(\alpha , \beta ,N^*\gamma)+(\mathcal{L}_{NX}\Phi )(\alpha , \beta ,\gamma)\\
                                              &\quad &\quad +\Phi ((\mathcal{L}_XN^*)\alpha , \beta ,\gamma)+\Phi (\alpha , (\mathcal{L}_XN^*)\beta ,\gamma)\\
                                              &=&-(N\iota _{\alpha \wedge \beta }\mathcal{L}_X\Phi -\iota _{\alpha \wedge \beta }\mathcal{L}_{NX}\Phi \\
                                              &\quad &\quad \quad \quad \quad \quad \quad \quad \quad -\iota _{(\mathcal{L}_XN^*)(\alpha \wedge \beta )}\Phi )(\gamma ).
\end{eqnarray*}
Hence, under the assumption of (\ref{(ii)}), it follows that $d_\pi ^2=[\Phi , \cdot ]_N$ on $\mathfrak{X}(M)$ if and only if the equality (\ref{(iii)}) holds. 

Since $d_\pi ^2$ and $[\Phi , \cdot ]_N$ are derivatives on $(\Gamma (\Lambda ^*TM),\wedge )$, it follows that $d_\pi ^2=[\Phi ,\cdot ]_N$ on $C^\infty (M)\oplus \mathfrak{X}(M)$ if and only if  $d_\pi ^2=[\Phi ,\cdot ]_N$ on $\mathfrak{X}^* (M)$. Therefore ii) is equivalent to that (\ref{(ii)}) and (\ref{(iii)}) hold.

Finally, iii) is equivalent to (\ref{(i)}) due to that $d_\pi \Phi =[\pi ,\Phi ]$. Therefore the proof has been completed.\qedhere

\vspace{3mm}
By the theorem, we have the following result of Kosmann-Schwarzbach \cite{K}.

\begin{corollary}
Under the same assumption as Theorem \ref{pPN-to-qLbia}, the triple $(M,\pi ,N)$ is a Poisson Nijenhuis manifold if and only if $((TM)_N, d_\pi )$ is a Lie bialgebroid.
\end{corollary}

As in the case of Poisson quasi-Nijenhuis Lie algebroids \cite{CNN}, we can consider a straightforward generalization of pseudo-Poisson Nijenhuis manifolds.

\begin{defn}
A {\it pseudo-Poisson Nijenhuis Lie algebroid} $(A, \pi, N, \Phi )$ is a Lie algebroid $A$ equipped with a $2$-section $\pi $ in $\Gamma (\Lambda ^2A)$, a Nijenhuis structure $N:A\longrightarrow A$ compatible with $\pi $ in the sense of Definition \ref{compatible} and a $3$-section $\Phi $ in $\Gamma (\Lambda ^3A)$ satisfying the conditions (\ref{(i)}), (\ref{(ii)}) and (\ref{(iii)}) replaced $[\cdot ,\cdot ]$ and $\mathcal{L}$ with $[\cdot ,\cdot ]_A$ and $\mathcal{L}^A$, respectively.
\end{defn}

\begin{theorem}
If a quadruple $(A, \pi, N, \Phi )$ is a pseudo-Poisson Nijenhuis Lie algebroid, then $(A_N, d_\pi , \Phi)$ is a quasi-Lie bialgebroid, where $A_N$ is a Lie algebroid deformed by the Nijenhuis structure $N$.
\end{theorem}

Now we show three simple and important examples of pseudo-Poisson Nijenhuis manifolds.

\begin{example}
A triple $(\pi ,N,\Phi )$, where $\Phi=0$, is a pseudo-Poisson Nijenhuis structure if $(\pi ,N)$ is a Poisson-Nijenhuis structure.
\end{example}

\begin{example}
Let $(M, \pi )$ be a Poisson manifold and set $N=0$. For any $d_\pi $-closed $3$-vector field $\Phi $, the triple $(\pi , N, \Phi )$ is a pseudo-Poisson Nijenhuis structure. Therefore, by Theorem \ref{pPN-to-qLbia} and Example \ref{qLbia-to-Ca}, $((TM)_N, d_\pi ,\Phi )$ is a quasi-Lie bialgebroid and $((TM)_N\oplus (T^*M)_\pi ,\langle \cdot ,\cdot \rangle ,[\![\cdot ,\cdot ]\!]_\pi ^\Phi ,\rho )$ is a Courant algebroid, where the Courant bracket $[\![\cdot ,\cdot ]\!]_\pi ^\Phi$ is defined by
\begin{eqnarray*}
&[\![X ,Y ]\!]_\pi ^\Phi =[X,Y]_0=0,\\
&[\![\xi ,\eta ]\!]_\pi ^\Phi =[\xi ,\eta ]_\pi +\Phi (\xi ,\eta ,\cdot ),\\
&[\![X, \xi ]\!]_\pi ^\Phi =(\iota _Xd_0\xi +\frac{1}{2}d_0<\xi ,X>)\\
&\quad \quad \quad \quad \quad \quad -(\iota _\xi d_\pi X+\frac{1}{2}d_\pi <\xi ,X>)\\
&\quad \quad \quad \ =-\iota _\xi d_\pi X-\frac{1}{2}d_\pi <\xi ,X>,
\end{eqnarray*}
the anchor map $\rho $ satisfies $\rho (X+\xi )=NX+\pi ^\sharp \xi =\pi ^\sharp \xi $ and the pairing $\langle \cdot ,\cdot \rangle $ is given by (\ref{pairing}) for any $X,Y$ in $\mathfrak{X}(M)$, any $\xi $ and $\eta $ in $\Omega ^1(M)$. 
\end{example}

\begin{example}\label{essential}
Let $M$ be a $C^\infty $-manifold and set $N=a\cdot \mathrm{id}_{TM}$, where $a$ is a non-zero real number. For any $2$-vector field $\pi $ in $\mathfrak{X}^2 (M)$, the triple $(\pi , N, \Phi )$, where $\Phi =\frac{1}{2a}\left[\pi , \pi \right]$, is a pseudo-Poisson Nijenhuis structure. Therefore $((TM)_N, d_\pi ,\Phi )$ is a quasi-Lie bialgebroid and $(TM\oplus T^*M ,\langle \cdot ,\cdot \rangle ,[\![\cdot ,\cdot ]\!]_\pi ^\Phi ,\rho )$ is a Courant algebroid, where the Courant bracket $[\![\cdot ,\cdot ]\!]_\pi ^\Phi$ is defined by
\begin{eqnarray*}
&[\![X ,Y ]\!]_\pi ^\Phi  =[X,Y]_{a\cdot \mathrm{id}_{T\! M}}=a[X,Y],\\
&[\![\xi ,\eta ]\!]_\pi ^\Phi =[\xi ,\eta ]_\pi +\frac{1}{2a}\left[\pi , \pi \right](\xi ,\eta ,\cdot ),\\
&\quad [\![X, \xi ]\!]_\pi ^\Phi =(\iota _Xd_{a\cdot \mathrm{id}_{T\! M}}\xi +\frac{1}{2}d_{a\cdot \mathrm{id}_{T\! M}}<\xi ,X>)\\
&\quad \quad \quad \quad \quad \quad \quad \quad \quad \quad -(\iota _\xi d_\pi X+\frac{1}{2}d_\pi <\xi ,X>)\\
&=a(\iota _Xd\xi +\frac{1}{2}d<\xi ,X>)\\
&\quad \quad \quad \quad \quad \quad -(\iota _\xi d_\pi X+\frac{1}{2}d_\pi <\xi ,X>),
\end{eqnarray*}
the anchor map $\rho $ satisfies $\rho (X+\xi )=aX+\pi ^\sharp \xi $ and the pairing $\langle \cdot ,\cdot \rangle $ is given by (\ref{pairing}) for any $X,Y$ in $\mathfrak{X}(M)$, $\xi $ and $\eta $ in $\Omega ^1(M)$.
\end{example}

Example \ref{essential} is an example of not a Poisson Nijenhuis manifold but a pseudo-Poisson Nijenhuis manifold.

The following proposition means that two given pseudo-Poisson Nijenhuis  manifolds generate a new one.

\begin{prop}
Let $(M_i,\pi _i,N_i,\Phi _i)$, $i=1,2$, be pseudo-Poisson Nijenhuis manifolds. Then the product $(M_1\times M_2, \pi _1+\pi _2, N_1\oplus N_2, \Phi _1+\Phi _2)$ is a pseudo-Poisson Nijenhuis manifold.
\end{prop}

\begin{proof} Using the fact that $[X_1,X_2]=0$ for any $X_i$ in $\mathfrak{X}(M_i),\ i=1,2$, etc., we can see that the triple $(\pi _1+\pi _2, N_1\oplus N_2, \Phi _1+\Phi _2)$ satisfies that $N_1\oplus N_2$ is a Nijenhuis structure on $M_1\times M_2$, the compatibility of $(\pi _1+\pi _2, N_1\oplus N_2)$ and the conditions (\ref{(i)}), (\ref{(ii)}) and (\ref{(iii)}) of Definition \ref{pPN}. \qedhere
\end{proof}

\section{Pseudo-symplectic Nijenhuis manifolds}\label{Pseudo-symplectic Nijenhuis manifolds}

In this section, we always assume that a $2$-vector field $\pi $ is nondegenerate. Then we can reduce one of the conditions for a triple $(\pi ,N,\Phi )$ to be a pseudo-Poisson Nijenhuis structure. This fact is important in the sense to be able to find pseudo-Poisson Nijenhuis structures easily. Moreover we rewrite a pseudo-Poisson Nijenhuis structure $(\pi ,N,\Phi )$ of which the $2$-vector field $\pi$ is nondegenerate using differential forms, and investigate properties of the structure. 

\begin{theorem}\label{nondeg}
Let $\pi $ be a nondegenerate $2$-vector field, $N$ a Nijenhuis structure compatible with $\pi $, and $\Phi $ a $3$-vector field. If a triple $(\pi ,N,\Phi )$ satisfies the conditions (\ref{(i)}) and (\ref{(ii)}) in Definition \ref{pPN}, then $(\pi ,N,\Phi )$ is a pseudo-Poisson Nijenhuis structure, i.e., $(\pi ,N,\Phi )$ satisfies the condition (\ref{(iii)}).
\end{theorem}

\proof We shall prove (\ref{(iii)}). By the nondegeneracy of $\pi $, the map $\pi ^\sharp :T^*M\longrightarrow TM$ is a bundle isomorphism. Therefore a set $\{\pi ^\sharp df| \ f\in C^\infty (M)\}$ generates locally the vector fields $\mathfrak{X}(M)$ as a $C^\infty (M)$-module. We have proved in Theorem \ref{pPN-to-qLbia} that the equality (\ref{(ii)}) holds if and only if $d_\pi ^2=[\Phi , \cdot ]$ holds on $C^\infty (M)$. Thus we compute, for any $f$ in $C^\infty (M)$,
\begin{eqnarray*}
d_\pi ^2(\pi ^\sharp df)&=&d_\pi ^2(-d_\pi f)=-d_\pi (d_\pi ^2f)=-d_\pi [\Phi ,f]_N\\
                        &=&-\left([d_\pi \Phi ,f]_N+[\Phi ,d_\pi f]_N\right)\\
                        &=&-[\Phi ,d_\pi f]_N=[\Phi ,\pi ^\sharp df]_N,
\end{eqnarray*}
where we use $\pi ^\sharp df=-d_\pi f$ in the first and the last step, the fourth equality follows from (\ref{pi_derivation}) and the fifth equality does from (\ref{(i)}). Therefore $d_\pi ^2=[\Phi , \cdot ]$ holds on the set $\{\pi ^\sharp df| \ f\in C^\infty (M)\}$. Since $d_\pi ^2=[\Phi , \cdot ]$ holds on $C^\infty (M)\oplus \{\pi ^\sharp df| \ f\in C^\infty (M)\}$ and since both $d_\pi ^2$ and $[\Phi , \cdot ]_N$ are derivatives on $(\Gamma (\Lambda ^*TM),\wedge )$, we obtain that $d_\pi ^2=[\Phi , \cdot ]$ holds on $\mathfrak{X}(M)$. This is equivalent to the condition (\ref{(iii)}) under the assumption of (\ref{(ii)}), so that the proof has been completed. \qedhere \\

In general, it is easier to deal with differential forms than multi-vector fields. Since a $2$-vector field $\pi $ is nondegenerate, there is a unique $2$-form $\omega $ corresponding with $\pi $. Hence it is convenience to transliterate conditions (\ref{(i)}) and (\ref{(ii)}) for $\pi $ into those for $\omega $. We compute
\begin{eqnarray}
&\frac{1}{2}[\pi ,\pi ](\alpha ,\beta ,\gamma )=-d\omega(\pi ^\sharp\alpha ,\pi ^\sharp\beta ,\pi ^\sharp\gamma ),\\
&<N\iota _{\alpha \wedge \beta }\Phi ,\gamma >=<N^*\iota _{\pi ^\sharp \alpha \wedge \pi ^\sharp \beta }(\omega ^\flat \Phi ),\pi ^\sharp \gamma >
\end{eqnarray}
for any $\alpha ,\beta $  and $\gamma $ in $\Omega^1(M)$, where a bundle map $\omega ^\flat : TM\longrightarrow T^*M$ is defined by $<\omega^\flat X,Y>:=\omega (X,Y)$. Therefore setting $\phi: =-\omega ^\flat \Phi $, we obtain the equivalence of the condition (\ref{(ii)}) and
\begin{eqnarray}\label{(ii)'}
\iota _{X\wedge Y}d\omega=N^*\iota _{X\wedge Y}\phi\ (X,Y \in \mathfrak{X}(M))
\end{eqnarray}
due to the nondegeneracy of $\pi $. Under the assumption of (\ref{(ii)'}), we calculate 
\begin{eqnarray}\label{(i)'}
[\pi ,\Phi ](\alpha _1,\alpha _2,\alpha _3,\alpha _4)=-(d\phi )(\pi ^\sharp \alpha _1,\pi ^\sharp \alpha _2,\pi ^\sharp \alpha _3,\pi ^\sharp \alpha _4)
\end{eqnarray}
for any $\alpha _i$ in $\Omega ^1(M)$. From the above, we see that the conditions (\ref{(i)}) and (\ref{(ii)}) are equivalent to the condition (\ref{(ii)'}) and the closedness of $\phi $ if $\pi $ is nondegenerate. Therefore we define as follows:

\begin{defn}\label{pseudo-symplectic Nijenhuis manifold}
Let $M$ be a $C^\infty $-manifold, $\omega $ a nondegenerate $2$-form on $M$, a $(1,1)$-tensor $N$ a Nijenhuis structure compatible with the nondegenerate $2$-vector field $\pi $ corresponding to $\omega $, and $\phi $ a closed $3$-form on $M$. Then a triple $(\omega , N, \phi )$ is a {\it pseudo-symplectic Nijenhuis structure} on $M$ if the condition (\ref{(ii)'}) holds. The quadruple $(M, \omega , N, \phi )$ is called a {\it pseudo-symplectic Nijenhuis manifold}. Obviously, $(M, \omega , N, \phi )$ is a pseudo-symplectic Nijenhuis manifold if and only if $(M, \pi ,N, \pi ^\sharp \phi )$ is a pseudo-Poisson Nijenhuis manifold.
\end{defn}

The following corollary states that we can construct new pseudo-symplectic Nijenhuis structures from a {\it symplectic Nijenhuis structure} $(\omega ,N)$, i.e., a pair $(\pi ,N)$ is a Poisson Nijenhuis structure, where $\pi $ is a nondegenerate Poisson structure corresponding to the symplectic structure $\omega $.

\begin{corollary}\label{s-N}
Let $(M,\omega ,N)$ be a symplectic Nijenhuis manifold and $\phi $ a closed $3$-form satisfies $\iota _{N\!X}\phi =0$ for any $X$ in $\mathfrak{X}(M)$. Then $(M,\omega ,N,\phi )$ is a pseudo-symplectic Nijenhuis manifold.
\end{corollary}

\proof In this case, the condition (\ref{(ii)'}) to prove is
\begin{eqnarray}\label{symp_(ii)'}
N^*\iota _{X\wedge Y}\phi =0\ (X,Y\in \mathfrak{X}(M))
\end{eqnarray}
because of $d\omega =0$. By computing that, for any $Z$ in $\mathfrak{X}(M)$,
\begin{eqnarray*}
<N^*\iota _{X\wedge Y}\phi ,Z>
                              =(\iota _{NZ}\phi )(X,Y)=0,
\end{eqnarray*}
where we use $\iota _{NX}\phi =0$, we conclude that (\ref{symp_(ii)'}) holds. Hence $(\omega ,N,\phi )$ is a pseudo-symplectic Nijenhuis structure. \qedhere \\

\begin{example}
On the $6$-torus $\mathbb{T}^6$ with angle coordinates $(\theta _1, \theta _2, \theta _3, \theta _4, \theta _5, \theta _6)$, we consider the standard symplectic structure $\omega :=d\theta _1\wedge d\theta _2+d\theta _3\wedge d\theta _4+d\theta _5\wedge d\theta _6$
 and a regular Poisson structure with rank $2$,
\begin{eqnarray*}
\pi_{\lambda } :=\frac{\partial}{\partial \theta _a}\wedge \left(\frac{\partial}{\partial \theta _b}+\lambda \frac{\partial}{\partial \theta _c}\right),
\end{eqnarray*}
where $\lambda $ is in $\mathbb{R}$ and $a, b$ and $c$ are three distinct numbers(see \cite{LM}). Setting $N_{\lambda }:=\pi_{\lambda } ^\sharp \circ \omega ^\flat$, we obtain a symplectic Nijenhuis structure $(\omega , N_{\lambda })$ on $\mathbb{T}^6$ (see \cite{V2} for a general theory of constructing symplectic Nijenhuis structures from symplectic and Poisson structures). Since the rank of $N_{\lambda }$ is $2$ at each points, the kernel of $N_{\lambda }^*$ is a subbundle with rank $4$ of the cotangent bundle of $\mathbb{T}^6$. Hence for any closed $3$-form $\phi $ in $\Gamma (\Lambda ^3\mathrm{Ker}N_{\lambda }^*)$, a triple $(\omega , N_{\lambda },\phi )$ is a pseudo-symplectic Nijenhuis structure on $\mathbb{T}^6$ by Corollary \ref{s-N}.
\end{example}

The following simple example is of a pseudo-symplectic Nijenhuis structure but not of a symplectic Nijenhuis structure.

\begin{example}
Let $(x^1,x^2.x^3,x^4)$ be the canonical coordinates in $\mathbb{R}^4$ and $f,g$ in $C^\infty (\mathbb{R})$ not constant but non-vanishing functions. We set
\begin{eqnarray*}
N:=\left( 
\begin{matrix}
 N^1_1 & \frac{(N^1_1-N^3_3)^2}{N^1_2} & 0     & 0\\
 N^1_2 & N^1_1                         & 0     & 0\\
 0     & 0                             & N^3_3 & \frac{(N^1_1-N^3_3)^2}{N^3_4}\\
 0     & 0                             & N^3_4 & N^3_3\\
\end{matrix}
\right),
\end{eqnarray*}
where $N^i_j$'s are in $\mathbb{R}^\times$ and  satisfy that $N^1_1\neq N^3_3$,
\begin{eqnarray*}
\omega:=f(a_3x^3+a_4x^4)dx^1\wedge dx^2+g(a_1x^1+a_2x^2)dx^3\wedge dx^4,
\end{eqnarray*}
where $a_i$'s satisfy $a_3:a_4=N^3_4:(N^1_1-N^3_3)$ and $a_1:a_2=N^1_2:(N^1_1-N^3_3)$, and
\begin{eqnarray*}
\phi \!&:=&\!(N^1_1)^{-1}f'(a_3x^3+a_4x^4)dx_1\wedge dx_2\wedge (a_3dx^3+a_4dx^4)\\
&\quad &\quad \quad \quad \quad +(N^3_3)^{-1}g'(a_1x^1+a_2x^2)(a_1dx^1+a_2dx^2)\wedge dx_3\wedge dx_4.
\end{eqnarray*}
Then $(\omega ,N,\phi )$ is a pseudo-symplectic Nijenhuis structure on $\mathbb{R}^4$.
\end{example}

Finally we describe a proterty of pseudo-symplectic Nijenhuis structures. This is the main theorem in this section.

\begin{theorem}
Let $(\omega ,N,\phi )$ be a pseudo-symplectic Nijenhuis structure on $M$ and $\pi $ the nondegenerate $2$-vector field corresponding to $\omega $. Then $(\pi _N,\phi )$ is a twisted Poisson structure \cite{SW}, i.e., the pair satisfies
\begin{eqnarray*}
&\frac{1}{2}[\pi _N,\pi _N]=\pi _N^\sharp \phi ,\\
&d\phi =0. 
\end{eqnarray*}
\end{theorem}

\begin{proof}
By Definition \ref{pseudo-symplectic Nijenhuis manifold}, we obtain $d\phi =0$. By Theorem \ref{pPN-to-qLbia}, $((TM)_N,d_\pi ,\Phi)$, where $\Phi:=\pi ^\sharp \phi$, is a quasi-Lie bialgebroid. Because of Proposition 4.8 in \cite{ILX}, the $2$-vector field $\pi _M$ on $M$ induced by the $2$-differential $d_\pi $ on $(TM)_N$ satisfies $\frac{1}{2}[\pi _M,\pi _M]=N\Phi $. Moreover, since we see that $\pi _M$ coincides with $\pi _N$ using Lemma 2.32 in \cite{ILX}, we have
\begin{eqnarray*}
\frac{1}{2}[\pi _N,\pi _N]=N\Phi =N(\pi ^\sharp \phi )=\pi _N^\sharp \phi .
\end{eqnarray*}
Therefore $(\pi _N, \phi )$ is a twisted Poisson structure on $M$. \qedhere 
\end{proof}

The property of a pseudo-symplectic Nijenhuis structure can be considered to be a generalization of the first step of the hierarchy of a Poisson Nijenhuis structure since the pair $(\pi _N,N)$ is compatible due to Proposition \ref{hierarchy}. We can obtain integrable systems by the hierarchy of a Poisson Nijenhuis structure. It is interesting to find apprications of psudo-symplectic (or of course, pseudo-Poisson) Nijenhuis structures to integrable systems.

\end{document}